\documentclass[a4paper,11pt]{article}
\usepackage[DIV=11]{typearea}
\usepackage{amsmath,amssymb,amsthm}
\usepackage[dvipsnames]{xcolor}
\usepackage[unicode, final, colorlinks=true]{hyperref}
\usepackage{textgreek}
\usepackage{enumitem, booktabs}
\usepackage{xcolor}
\usepackage{siunitx}
\usepackage{bbm}
\usepackage{mathtools}
\usepackage{tikz, pgfplots}
\usepackage{romanbar}

\hypersetup{citecolor=teal, linkcolor=BrickRed, urlcolor=NavyBlue}

\pgfplotsset{compat=1.17}
\usepgfplotslibrary{groupplots}
\usepgfplotslibrary{fillbetween}
\usetikzlibrary{backgrounds}
\usetikzlibrary{arrows.meta}
\pgfplotsset{plot coordinates/math parser=false}
\usetikzlibrary{external}
\usepgfplotslibrary{patchplots}
\usetikzlibrary{shapes.geometric}
\usetikzlibrary{backgrounds}
\usetikzlibrary{intersections}
\usetikzlibrary{spy}

\newcommand{\R}{\mathbb{R}}
\newcommand{\Z}{\mathbb{Z}}

\newcommand{\minmod}{\operatorname{minmod}}
\newcommand{\CFL}{\operatorname{CFL}}
\newcommand{\sign}{\operatorname{sign}}

\newcommand{\Rmn}[1]{\Romanbar #1}
\newcommand{\rmn}[1]{\text{\tiny \Romanbar{#1}}}

\providecommand{\keywords}[1]{\textit{Keywords:} #1}
\providecommand{\msc}[1]{\textit{2010 MSC:} #1}

\newtheorem{theorem}{Theorem}[section]

\newtheorem{proposition}[theorem]{Proposition}
\newtheorem{remark}[theorem]{Remark}

\begin{document}
\date{
  \small
  Institute of Geometry and Applied Mathematics,\\ RWTH Aachen University, Templergraben 55,\\ 52062 Aachen, Germany\\
   \smallskip
   {\tt  \{beckers,kolbe\}@igpm.rwth-aachen.de} \\
   \smallskip
   \today
  }

\title{The Lax--Friedrichs method in one-dimensional hemodynamics}
\author{Anika Beckers \and Niklas Kolbe}

\maketitle

\begin{abstract}
  The discretization of reduced one-dimensional hyperbolic models of blood flow using the Lax--Friedrichs method is discussed. Employing the well-established central scheme in this domain significantly simplifies the implementation of specific boundary and coupling conditions in vascular networks accounting e.g.\ for a periodic heart beat, vascular occlusions, stented vessel segments and bifurcations. In particular, the coupling of system extensions modeling patient specific geometries and therapies can be realized without information on the eigenstructure of the models. For the derivation of the scheme and the coupling conditions a relaxation of the model is considered and its discrete relaxation limit evaluated. Moreover, a second order MUSCL-type extensions of the scheme is introduced. Numerical experiments in uncoupled and coupled cases that verify the consistency and convergence of the approach are presented.
  
  \bigskip
  \noindent \keywords{blood flow modeling; cardiovascular networks; finite volumes; coupled conservation laws; boundary conditions; hyperbolic systems}\\
  \msc{35L65, 35R02, 00A71, 62P10}
  \end{abstract}

  \section{Introduction}
  Over the last decades the number of cardiovascular disease cases in Europe has significantly increased; it nowadays accounts for 45 \% of deaths in Europe~\cite{wilkins2017europcardiovdiseasstatis}.
  Among the concerned medical conditions stroke has been found to be responsible for most of the fatalities \cite{saini2021globalepidemstrok}. Surgical treatments such as stent placements, coronary artery bypass grafting and endovascular thrombectomy relies on information about the patient specific hemodynamics, which accounts for the blood flow through the vasculature and the fluid-structure interaction with the vessel walls, see e.g.\ \cite{balossino2008effec,forti2019transopticmonit}.  

  Computational models have been shown to be useful tools for treatment development and operation planning, see \cite{crosetto2011fluid, neidlin2016investhemo} and the references therein. While full-scale models in three space dimensions have been well-established \cite{nobile2008robin} these models suffer from complexity and long run-times hampering clinical application. Reduced one-dimensional models based on simplifying assumptions on geometry and flow profiles, blood flow and interactions with the vessel wall have offered an efficient alternative~\cite{hughes1973,formaggia2003}. Those have lately been increasingly used in applications ranging from uncertainty quantification in \cite{fleeter2020multil} to simulating clinical interventions using intravascular catheters in \cite{benemerito2023,pradhan2024}.

  Employing these mostly hyperbolic models in real-time applications requires efficient numerical schemes. In the literature simulations have been carried out by Taylor--Galerkin, cf.\ \cite{donea1984time}, or Godunov-type finite volume schemes, see e.g., \cite{melis2019improved,xiaofei2015verif,formaggia2006,peiro2009reduc}. This work is concerned with the application of the Lax--Friedrichs method to cardiovascular models in one space dimension. The Lax--Friedrichs method introduced in \cite{friedrichs1971system} belongs to the class of central schemes, cf.\ \cite{kurganov2000newhighresol}, and has been a popular choice for the solution of hyperbolic problems. Its main advantage is its universality, which is due to the fact that it is not tied to the eigenstructure of the discretized problem. In the context of cardiovascular models this allows for a new simplified handling of domain boundaries and coupling conditions at vascular junctions in a network, which in Godunov-type schemes rely on the Lax-curves corresponding to the model. It facilitates the coupling, in particular, in model extensions accounting for blood solutes and endoscopic therapy for which Lax-curves might not be available. While this work introduces the general application and the correct handling of boundaries and network nodes, the companion paper \cite{herty2024} applies the approach to the modeling of aspiration therapy.     

Typically, boundary conditions for the reduced one-dimensional models are based on the characteristic variables of the model and requires the extrapolation of the outgoing invariants and additional rules for the ingoing invariant \cite{formaggia2006,peiro2009reduc}. Conditions for the coupling of vessels at a junctions have been modeled in \cite{formaggia2003}. We also mention the recent approach from \cite{lucca2023}, in which the coupling has been modeled in more detail using a higher-dimensional sub-model. The coupling in our work and the derivation of the scheme will follow the approach from  \cite{herty2023centrschemtwo,herty2023centr} and employs the relaxation introduced in \cite{jin1995relaxschemsystem} for coupled hyperbolic systems. Discretizing the relaxed system using an implicit-explicit asymptotic preserving scheme recovers the Lax--Friedrichs scheme in the relaxation limit. Among others our approach allows us to couple the system for different velocity profiles encoded in different momentum-flux correction coefficients and to handle non-standard pressure models accounting for the viscoelasticity of the vessel wall.

The rest of the paper is structured as follows. In Section \ref{sec:model} we recall the reduced modeling of blood flow in a straight elastic vessel and thereby discuss two hyperbolic systems in different conserved variables. In Section~\ref{sec: LF} we consider the Xin-Jin-type relaxation of the blood flow models and use it to derive a Lax--Friedrichs type method. Also, a second-order extension using the MUSCL scheme is introduced, and an approximation is proposed that accounts for the diffusive term in the viscoelasticity considering model extension. 
In Section~\ref{sec:BC} we construct the correct boundary conditions for our approach. These do not require the extrapolation of the characteristic variables and can easily be extended to a higher order accuracy. Application specific boundary conditions result from imposing one of the key quantities, e.g.\ the pressure or the blood velocity, at the boundary.
Section~\ref{sec:cpl} progresses from the previously considered straight vessel to a vascular network by coupling various such vessels. A one-to-one coupling, which can deal with discontinuities in the vessel properties, and bifurcations in the vascular network are considered. The coupling conditions from \cite{formaggia2003} are adjusted for an application within the presented scheme, similar to the handling of the boundary.
Finally, in Section~\ref{sec:exp} numerical experiments are presented. Along with simulations of blood flow assuming entering pulse waves in the uncoupled and coupled case we verify our approach using grid convergence studies of our schemes with regard to the global spatial error and the coupling error at the interface.

\section{Modeling of blood flow in an elastic vessel}\label{sec:model}
In this section we discuss reduced one-dimensional models for blood flow in a single large vessel; vascular networks can be accounted for by coupling various straight parts, see e.g.,~\cite{formaggia2003}. The considered vessel segment is assumed cylindrical, and the flow is described by an average axial velocity component. This allows for a model with state variables only depending on the time $t$ and a single spatial variable $x$ parameterizing the axis of the vessel. The model is obtained describing blood as a Newtonian fluid and averaging the compressible Navier--Stokes equations over the cross-sections of the elastic cylindrical domain, see e.g.,~\cite{hughes1973} for details.

Denoting the cross-section area by $A=A(x,t)$ and the mass flux by $Q=Q(x,t)$, the model consists of the following two equations, accounting for mass conservation and momentum balance.
\begin{equation}\label{eq:systemQ}
	\begin{aligned}
		&\frac{\partial A}{\partial t} + \frac{\partial Q}{\partial x} = 0, \\
		& \frac{\partial Q}{\partial t} + \frac{\partial}{\partial x} \bigl(\alpha \frac{Q^2}{A} \bigr) + \frac{A}{\rho} \frac{\partial p}{\partial x} = S_v(A, Q).  &
	\end{aligned}
\end{equation}
In the second equation $p=p(x, t)$ denotes the space and time dependent pressure and $\rho$ the constant blood density.
The parameter $\alpha$ is the momentum-flux correction coefficient relating the averaged momentum to the actual momentum \cite{formaggia1999multismodelcirculsystem}, and $S_v$ is a source term accounting for the viscosity of blood. Both depend on the velocity profile within the cross-section; if we assume the profile to be independent of the axial position and radially symmetric, i.e. such that the radial component of the velocity denoted by $v$ depends only on the radial position $r$, we obtain
\begin{equation}
  S_v(A, Q) = 2 \mu R \left[ \frac{\partial v}{\partial r} \right]_{r=R}, \qquad R = \sqrt{\frac{A}{\pi}}
\end{equation}
with $\mu$ referring to the dynamic viscosity.

Commonly, Hagen--Poiseuille flow \cite{white1991viscous} is assumed with the velocity profile
\begin{equation}\label{eq:hpflow}
  v = u \frac{\psi + 2}{\psi} \left( 1 - \frac{r^\psi}{R^\psi}\right),
\end{equation}
where $u$ denotes the average axial velocity. Under this assumption the momentum-flux correction coefficient satisfies $\psi = (2-\alpha)/(\alpha -1)$. The choice $\psi=9$ (and thus $\alpha=1.1$) has been found to fit well to experimental data in \cite{smith2002}.

\paragraph{Vascular pressure.} The pressure exerted on the vessel is dominated by forces due to the vessel wall displacement \cite{peiro2009reduc}. This is reflected in the algebraic pressure law
\begin{equation}\label{eq:commonpressure}
	p  = P_{\text{ext}} + \beta (\sqrt{A}- \sqrt{A_0}),
\end{equation}
where $P_{\text{ext}}$ is a constant external pressure and $\beta= \frac{\sqrt{\pi}h_0 E}{(1-\nu^2)A_0}$ with $A_0$ denoting the reference section area and $h_0$ the wall thickness, respectively. Moreover, $\nu$ is the Poisson ratio and $E$ is the Young modulus corresponding to the vessel. A typical choice for the Poisson ratio is $\nu=0.5$, reflecting the incompressibility of the wall tissue \cite{formaggia2003}.
The pressure law \eqref{eq:commonpressure} assumes that the wall reacts immediately to compressive forces at the appropriate location. Additionally taking into account the viscoelasticity of the vessel wall by adopting a Voigt–Kelvin model~\cite{fung1994biomec} gives rise to the pressure
\begin{equation}\label{eq:pressurelong}
	\tilde p= P_{\text{ext}}  + \beta (\sqrt{A}- \sqrt{A_0}) + \frac{\gamma \sqrt{\pi}}{2\sqrt{A_0}^3} \frac{\partial A}{\partial t} 
\end{equation}
with $\gamma$ being the viscoelasticity coefficient. Assuming that $A_0$ and $\gamma$ are constant we can employ the first equation in \eqref{eq:systemQ} to rewrite the last term in \eqref{eq:pressurelong} and obtain a parabolic model, in which the term
\begin{equation}\label{eq:parabolicterm}
  \frac{\gamma \sqrt{\pi} A}{2 \rho \sqrt{A_0}^3} \frac{\partial^2 Q}{\partial x^2} 
\end{equation}
is added to the right-hand side of the second equation in \eqref{eq:systemQ}.

\paragraph{Velocity form.} As it holds $Q= Au$ for the mass flow basic algebra can be used to replace the \emph{flow form} \eqref{eq:systemQ} by rewriting the model in the \emph{velocity form}
\begin{equation}\label{eq:systemu}
	\begin{split}
          &\frac{\partial A}{\partial t} + \frac{\partial Au}{\partial x} = 0, \\
          & \frac{\partial u}{\partial t} + (2\alpha-1) u \frac{\partial u}{\partial x} + (\alpha-1) \frac{u^2}{A} \frac{\partial A}{\partial x} + \frac{1}{\rho} \frac{\partial p}{\partial x}  = \frac{1}{A} S_v(A, Au).
	\end{split}
\end{equation}
This model formulation has been particularly useful for the design of coupling conditions on a network, see e.g., \cite{formaggia2003}. While in the following we will base our study on the flow form~\eqref{eq:systemQ} implications for the velocity form \eqref{eq:systemu} will be discussed as well.

\paragraph{Conservative form.} Being derived from basic principles, model \eqref{eq:systemQ} admits a conservative form\footnote{Our notion of conservative systems here neglects the viscous source term $S_v$.}. This is made clear, using the antiderivative of the pressure function \eqref{eq:commonpressure}, i.e.,
\begin{equation*}
  P(A) = A_0 P_\text{ext} + \int_{A_0}^A p(a) \, da,
\end{equation*}
and writing the second equation in \eqref{eq:systemQ} as
\begin{equation}\label{eq:flow2conservative}
  \frac{\partial Q}{\partial t} + \frac{\partial}{\partial x} \left( \alpha \frac{Q^2}{A} + \frac{1}{\rho} (Ap - P)\right) = S_v(A, Q).
\end{equation}
The velocity form \eqref{eq:systemu} on the other hand does not always admit a conservative form, as the following proposition shows.
\begin{proposition}
  System~\eqref{eq:systemu} is conservative if and only if $\alpha=1$.
\end{proposition}
\begin{proof}
  If $\alpha=1$ we can write the second equation in \eqref{eq:systemu} as
  \begin{equation*}
    \frac{\partial u}{\partial t} + \frac{\partial}{\partial x} \left( \frac 12 u^2 +  \frac{p}{\rho} \right) =\frac{1}{A} S_v(A, Au)
  \end{equation*}
  and thus \eqref{eq:systemu} is conservative. Next, let $\alpha \ne 1$. We can generally rewrite  the second equation in \eqref{eq:systemu} as 
  \begin{equation}\label{eq:vel2variant}
    \frac{\partial u}{\partial t} + (2\alpha - 1)u \frac{\partial u}{\partial x} + \left( (\alpha -1) \frac{u^2}{A} + \frac{1}{\rho} \frac{\partial p}{\partial A} \right) \frac{\partial A}{\partial x}  =\frac{1}{A} S_v(A, Au).
  \end{equation}
  We assume in \eqref{eq:vel2variant} that neither $\beta$ nor $A_0$ vary in space. Next, we suppose that the system was conservative. Then there would be a smooth function $f$ in the variables $A$ and $u$ such that
  \begin{equation*}
    \frac{\partial u}{\partial t} + \frac{\partial f(A,u)}{\partial x} = \frac{1}{A} S_v(A, Au).
  \end{equation*}
  Due to \eqref{eq:vel2variant} it would hold
  \[
    T^A(A,u) \coloneqq (\alpha -1) \frac{u^2}{A} + \frac{1}{\rho} \frac{\partial p}{\partial A} = \frac{\partial f(A,u)}{\partial A}, \qquad T^u(A,u) \coloneqq (2\alpha - 1)u = \frac{\partial f(A,u)}{\partial u}.
  \]
      But since
      \[
        \frac{\partial^2 f(A,u)}{\partial A \partial u} =  \frac{\partial^2 f(A,u)}{\partial u \partial A}
      \]
      and
      \[
        \frac{\partial  T^A(A,u)}{\partial u} = 2 (\alpha -1) \frac{u}{A} \ne 0 = \frac{\partial T^u(A,u)}{\partial A}
      \]
      we would arrive at a contradiction. Thus the system cannot be conservative. 
    \end{proof}
    \begin{remark}
      If $\alpha=1$ then system \eqref{eq:systemu} would also be conservative for more general pressure functions, e.g., in the case that $p$ depends also on $u$. Note that in this case \eqref{eq:systemQ} might not be conservative.
    \end{remark}
    In this work we assume that the reference section area as well as the wall parameter $\beta$ are constant. We note though that smooth spatial variation in these parameters can be accounted for using additional source terms in the second equation in \eqref{eq:systemQ} or \eqref{eq:systemu}, respectively, which do not affect the conservation properties.

\section{The Lax--Friedrichs method as a relaxation limit}\label{sec: LF}
In this section we consider a relaxation of the blood flow model from Section~\ref{sec:model} along with a discretization that recovers the Lax--Friedrichs scheme in the relaxation limit.

\subsection{The relaxation system}
For the sake of brevity, the following analysis relies on the vector notation 
\begin{equation}\label{eq:systemQvector}
  \mathbf U = (A, Q)^T, \qquad \mathbf F(\mathbf U) = \left(Q, \alpha \frac{Q^2}{A} + \frac{1}{\rho} \left(Ap - P \right)\right)^T, \qquad \mathbf S(\mathbf U) = (0, S_v(A, Q))^T,
\end{equation}
which allows us to express model \eqref{eq:systemQ} as the $2 \times 2$ system of balance laws
\begin{equation}\label{eq:vecsystem}
\frac{\partial \mathbf U}{\partial t} + \frac{\partial \mathbf F(\mathbf U)}{\partial x} = \mathbf S(\mathbf U).
\end{equation}
Introducing now the new variable $\mathbf V = (V^A, V^Q)^T$ mapping the time and space variables to $\R^2$ we follow \cite{jin1995relaxschemsystem} to obtain for any $\varepsilon>0$ the Jin--Xin-type relaxation system
\begin{equation}\label{eq:relaxation}
  \begin{aligned}
    \frac{\partial \mathbf U}{\partial t} + \frac{\partial  \mathbf V}{\partial x} &= \mathbf S(\mathbf U), \\
   \frac{ \partial \mathbf V}{\partial t} + \lambda^2 \, \frac{ \partial \mathbf U}{\partial x} &= \frac{1}{\varepsilon}(\mathbf F(\mathbf U) - \mathbf V).
\end{aligned}
\end{equation}
If the relaxation speed $\lambda$ is chosen as an upper bound of the system Jacobian, in particular
\begin{equation}\label{eq:lambdastability}
\lambda \geq \alpha \frac{Q}{A} \pm \sqrt{\alpha(\alpha - 1) \frac{Q^2}{A^2} + \frac{\beta}{2 \rho} \sqrt{A}},
\end{equation}
the subcharacteristic condition verifying the stability of \eqref{eq:relaxation} holds, see \cite{liu1987hyperconserlawsrelax} for details. In the asymptotic relaxation limit $\varepsilon \to 0$ the auxiliary variable $\mathbf V$ approaches $\mathbf F (\mathbf U)$ and $\mathbf U$ solves the original problem \eqref{eq:vecsystem}, see \cite{chen1994hyperconserlaws}. Numerical schemes for \eqref{eq:relaxation} that preserve this limit property have been of high interest \cite{hu2017asymppreserschem}. The unsplit scheme proposed in \cite{jin2012asympap} is such a scheme, which we consider in more details.

\subsection{Discretization of the relaxation system.} Let a uniform partition of the real line into the mesh cells $I_j=(x_{j-1/2}, x_{j+1/2})$ of width $\Delta x$ be given. Additionally, we consider the time instances $t^n= n \Delta t$ for the uniform time step size $\Delta t$ satisfying the condition
\begin{equation}\label{eq:time-discretization}
  \Delta t= \CFL \frac{\Delta x}{\lambda}
\end{equation}
and a suitable Courant number $0<\CFL \leq 1$. We approximate the vector valued states of \eqref{eq:relaxation} by volume averages over the cell $I_j$ at time $t^n$ that we denote as $\mathbf U_j^n$ and $\mathbf V_j^n$, and adopt an analogue notation for its components. Then the scheme reads
\begin{equation}\label{eq:unsplit}
\begin{aligned}
  \mathbf U_j^{n+1} &= \mathbf U_j^n - \frac{\Delta t}{ 2 \Delta x}\left( \mathbf V^{n}_{j+1} - \mathbf V^{n}_{j-1}\right) + \frac{\lambda \Delta t}{2 \Delta x} \left(\mathbf U_{j+1}^n - 2 \mathbf U_j^n + \mathbf U_{j-1}^n \right) + \Delta t S(\mathbf U_j^n), \\
  \mathbf V_j^{n+1} &= \mathbf V_j^{n} - \frac{\lambda^2 \Delta t}{ 2 \Delta x}\left( \mathbf U^{n}_{j+1} - \mathbf U^{n}_{j-1}\right) + \frac{\lambda \Delta t}{2 \Delta x} \left(\mathbf V_{j+1}^{n} - 2 \mathbf V_j^{n} + \mathbf V_{j-1}^{n} \right) \\
  &\quad + \frac{\Delta t}{\varepsilon} \left(\mathbf F (\mathbf U_j^{n+1}) - \mathbf V_j^{n+1}\right).
\end{aligned}
\end{equation}
We note that the scheme has been derived from an upwind discretization of \eqref{eq:relaxation} in characteristic variables and an implicit-explicit time discretization. The implicit time discretization in the second equation of \eqref{eq:unsplit} is necessary to obtain an asymptotic preserving scheme, cf.~\cite{hu2017asymppreserschem}, but does not require the solution of a nonlinear system as $\mathbf U_j^{n+1}$ can be computed independently.

\paragraph{The limit scheme.} As we take the limit $\varepsilon \to 0$ in \eqref{eq:unsplit} we obtain the scheme
\begin{equation}\label{eq:limitscheme}
  \mathbf U_j^{n+1} = \mathbf U_j^n - \frac{\Delta t}{\Delta x}\left( \mathcal F_{j+1/2}^n - \mathcal F_{j-1/2}^n \right) + \Delta t S(\mathbf U_j^n)
\end{equation}
for the original problem \eqref{eq:vecsystem} with numerical fluxes given by
\begin{equation}\label{eq:numflux}
  \mathcal F_{j-1/2}^n = \frac 1 2 \left( \mathbf V_{j-1}^n + \mathbf V_{j}^n \right) - \frac \lambda 2 \left( \mathbf U_j^n - \mathbf U_{j-1}^n \right), \qquad \mathbf V_j^n = \mathbf F (\mathbf U_j^n) \quad \forall j \in \Z.
\end{equation}
Refer to \cite{herty2023centrschemtwo} for the detailed limit procedure. Taking into account \eqref{eq:time-discretization} the classical Lax--Friedrichs method is recovered taking $\CFL=1$.

\paragraph{High resolution extension.}
To increase the accuracy in space we use the MUSCL scheme~\cite{leer1979towarultimconserdifferschem} that employs slope reconstructions. Here we apply the approach to the characteristic variables of the relaxation system \eqref{eq:relaxation} implying that four scalar quantities are linearly reconstructed. This discretization gives rise to the high order correction terms
\begin{equation}\label{eq:muscl}
\mathcal H_{j-1/2}^\text{MUSCL} = \frac{\Delta x}{2} (\mathbf s_{j-1}^{n,-} - \mathbf s_{j}^{n,+}) \quad \forall j \in \Z
\end{equation}
to be added to the numerical fluxes \eqref{eq:numflux}.
We use the minmod limiter to prevent oscillatory behavior of the scheme. The reconstructed slopes in \eqref{eq:muscl} thus take the form
\begin{align}
  \mathbf s_j^{n, \pm} &\coloneqq \minmod\left( \frac{\mathbf V_j^n  - \mathbf V_{j-1}^n \pm  \lambda (\mathbf U_j^n - \mathbf U_{j-1}^n)  }{2 \Delta x},
  \frac{\mathbf V_{j+1}^n - \mathbf V_j^n \pm \lambda (\mathbf U_{j+1}^n - \mathbf U_{j}^n)  }{2 \Delta x} \right),
\end{align}
where the \emph{minmod} operator is given by
\[
  \minmod(a,b) =
  \begin{cases*}
    0 & if $\sign(a) \neq \sign(b)$ \\
    a & if $|a| \leq |b|$ and $\sign(a) = \sign(b)$ \\
    b & if $|a| > |b|$ and $\sign(a) = \sign(b)$
  \end{cases*}
\]
for scalar arguments $a,b \in \R$ and component-wise in case of vectors.

\subsection{Implications for the velocity form and extended pressure models}
In this section we sketch how our relaxation based scheme derivation can be applied to model \eqref{eq:systemu} and how the extended pressure law \eqref{eq:pressurelong} is treated numerically. 
\paragraph{Velocity form.} The above approach is also applicable to the blood flow model in velocity form \eqref{eq:systemu} in the case $\alpha=1$ by taking
\begin{equation}\label{eq:systemuvector}
  \mathbf U = (A, u)^T, \qquad \mathbf F(\mathbf U) = \left(Au, \frac 1 2 u^2 + \frac{p}{\rho}\right)^T, \qquad \mathbf S(\mathbf U) = \left(0, \frac{1}{A}S_v(A, Q)\right)^T
\end{equation}
instead of \eqref{eq:systemQvector} in \eqref{eq:vecsystem}. To satisfy the subcharacteristic condition in this case, it is sufficient to choose the relaxation speed $\lambda$ such that \eqref{eq:lambdastability} holds for $\alpha=1$. If $\alpha \ne 1$ the generalized relaxation approach from \cite{kolbe2024numerschemcoupl} relying on path-conservative schemes can be applied to system~\eqref{eq:systemu}.

\paragraph{Viscoelasticity.} In the case that viscoelasticity is considered within model \eqref{eq:systemQ} by means of the extended pressure law \eqref{eq:pressurelong} the additional second order term \eqref{eq:parabolicterm} is approximated by the semi-implicit finite difference formula  

\begin{equation}
  R_j(\mathbf U^n, \mathbf U^{n+1}) = \frac{\gamma \sqrt{\pi} A_j^n}{2 \rho \sqrt{A_0}^3} ~ \frac{Q_{j+1}^{n+1} - 2 Q_j^{n+1} + Q_{j-1}^{n+1}}{\Delta x^2}.
\end{equation}
To simulate the augmented model we add $(0, \Delta t  R_j(\mathbf U^n, \mathbf U^{n+1}))^T$  to the right-hand side of scheme \eqref{eq:limitscheme}, which results in a linear system that is to be solved in each time step. As the system matrix is tridiagonal we use the Thomas algorithm to efficiently solve the system within our scheme.

\section{Boundary conditions}\label{sec:BC}
\begin{figure}
  \centering
  \includegraphics{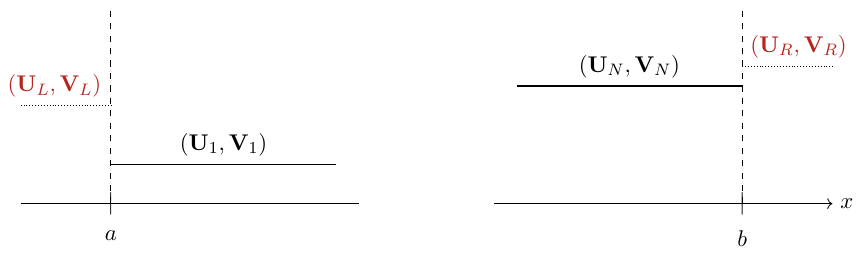}
  \caption{The relaxation system at the boundaries. The states $\mathbf U_L$, $\mathbf V_L$, $\mathbf U_R$ and $ \mathbf V_R$ constitute suitable boundary data with respect to the cell averages $\mathbf U_1$, $\mathbf V_1$, $\mathbf U_N$, $\mathbf V_N$ of the numerical solution.}\label{fig:boundarystates}
\end{figure}

In this section we discuss appropriate boundary conditions for scheme~\eqref{eq:limitscheme}. Boundary conditions for the blood flow model \eqref{eq:systemQ}, which can, among others, take into account a beating heart, the vascular periphery or occlusions, have been extensively discussed in the literature, see e.g.,~\cite{formaggia2006}. The new boundary conditions we construct here need to account for the fact that the employed Lax--Friedrichs scheme is derived from relaxation system \eqref{eq:relaxation} and thus constitute suitable boundary data for this system~\cite{dubois1988boundconditnonlin}. In particular, this means that not only boundary states with respect to the original state $\mathbf U$ but also with respect to the variable $\mathbf V$ are required. Although our problem is of parabolic nature if the extended pressure form \eqref{eq:pressurelong} is assumed we always assume that the problem remains convection dominated and derive boundary data from the hyperbolic part of the system.

In the following we consider blood flow on the bounded domain $(a,b)$ discretized by scheme \eqref{eq:limitscheme} over $N$ mesh cells $I_1,\ldots,I_N$ such that $a$ is located at the left boundary of $I_1$ and $b$ at the right boundary of $I_N$. Given the cell averages $\mathbf U_1^n, \dots, \mathbf U_N^n$ and fluxes $\mathbf V_1^n, \dots, \mathbf V_N^n$ the boundary states $\mathbf U_L^n$, $\mathbf V_L^n$, $\mathbf U_R^n$, $\mathbf V_R^n$ for the computation of $\mathcal F_{1/2}^n$ and $\mathcal F_{N+1/2}^n$ according to \eqref{eq:numflux} are sought in order to update the numerical solution to the time instance $t^{n+1}$. For brevity, we drop the time index in the following discussion. The situation is visualized in Figure~\ref{fig:boundarystates}. We make use of our findings from~\cite{herty2023centr} implying that suitable boundary data for the relaxation system satisfies the conditions 
\begin{equation}\label{eq:laxcurves}
  \mathbf V_L - \mathbf V_1 = \lambda (\mathbf U_L - \mathbf U_1)
  \qquad \text{and} \qquad
  \mathbf V_R - \mathbf V_N = \lambda (\mathbf U_N - \mathbf U_R).
\end{equation}

\subsection{Non-reflecting boundary conditions}
Typically, boundaries in the blood flow model are assumed to be non-reflecting, which implies that 
\begin{equation}\label{eq:nonreflecting}
  \mathbf l_1(\mathbf U_L)^T \, \frac{ \partial \mathbf F(\mathbf U)}{\partial x} \bigg \rvert_{x=a} = 0 \qquad \text{and} \qquad
  \mathbf l_2(\mathbf U_R)^T \, \frac{ \partial \mathbf F(\mathbf U)}{\partial x} \bigg \rvert_{x=b} = 0
\end{equation}
holds, where $\mathbf l_1$ and  $\mathbf l_2$ denote the left eigenvectors corresponding to the negative and positive eigenvalue of the system \eqref{eq:vecsystem}, respectively, see \cite{thompson1987time}. This has been addressed in the literature via extrapolation of the outgoing Riemann invariant within Taylor--Galerkin schemes, see \cite{quarteroni2004mathem}; we consider an alternative approach more suitable for our numerical method. In the situation shown in Figure~\ref{fig:boundarystates}, where we have piecewise constant data, the derivatives in \eqref{eq:nonreflecting} can be approximated by a central difference with respect to a small increment $\delta < \Delta x$ so that after eliminating the denominator we obtain
\begin{equation}\label{eq:nonreflectingdifference}
  \mathbf l_1(\mathbf U_L)^T [\mathbf F(\mathbf U_1) - \mathbf F(\mathbf U_L)]  = 0 \qquad \text{and} \qquad
  \mathbf l_2(\mathbf U_R)^T [\mathbf F(\mathbf U_R) - \mathbf F(\mathbf U_N)] = 0.
\end{equation}
Next, employing \eqref{eq:nonreflectingdifference} to formulate a consistent condition in the variable $\mathbf V$, cf.~\cite{herty2023centrschemtwo}, and using \eqref{eq:laxcurves} we derive the conditions
\begin{equation}\label{eq:nonreflectingQ}
  \begin{aligned}
    Q_1 - Q_L &= \left(\alpha \frac{Q_L}{A_L}+ \sqrt{\alpha (\alpha -1) \frac{Q_L^2}{A_L^2} + \frac{\beta}{2 \rho A_0} \sqrt{A_L}}~\right) \, (A_1 - A_L), \\
    Q_N - Q_R &= \left(\alpha \frac{Q_R}{A_R}- \sqrt{\alpha (\alpha -1) \frac{Q_R^2}{A_R^2} + \frac{\beta}{2 \rho A_0} \sqrt{A_R}} ~\right)\, (A_N - A_R)
  \end{aligned}
\end{equation}
for system \eqref{eq:systemQ}, and under the assumption $\alpha=1$ the conditions
\begin{equation}\label{eq:nonreflectingu}
  u_1 - u_L =  A_L^{-3/4} \sqrt{\frac{\beta}{2\rho}} (A_1 - A_L) 
  \qquad \text{and} \qquad
  u_N - u_R =  -A_L^{-3/4} \sqrt{\frac{\beta}{2\rho}} (A_N - A_R)  
\end{equation}
for system \eqref{eq:systemu}. In Appendix~\ref{sec:boundary2} we provide a second order approach to the boundary conditions \eqref{eq:nonreflectingQ}. 

\subsection{Prescribed pressure, mass flow and velocity at the boundary} For well-posedness at the boundary additional conditions complementing \eqref{eq:nonreflectingQ} or  \eqref{eq:nonreflectingu} are needed. In some application a boundary pressure $p_L$ or $p_R$ is imposed. In this case the boundary section area can be deduced inverting the pressure law \eqref{eq:commonpressure}, i.e.\ taking $A_L=p^{-1}(p_L)$ or $A_R = p^{-1}(p_R)$. The corresponding mass flow / velocity is then computed from \eqref{eq:nonreflectingQ} / \eqref{eq:nonreflectingu} solving a linear equation if $\alpha=1$ or a quadratic one after eliminating the square root if $\alpha \ne 1$.

In other applications a boundary velocity $u_L$ or $u_R$ might be imposed. If the model is given in velocity form and $\alpha=1$ then the corresponding section area at the boundary can be computed from \eqref{eq:nonreflectingu}, which after eliminating all fractional exponents involves the solution of a quartic equation. If the model is given in flow form we first rewrite $Q=Au$ in \eqref{eq:nonreflectingu} and afterwards similarly solve the corresponding equation for $A_L$ or $A_R$. In the special case of a reflecting boundary, i.e., $u_L=0$ or $u_R=0$, this computation significantly simplifies.

The above techniques are used for instance to model a beating heart. Therefore, one alternates between imposing a sinusoidal inlet pressure and reflecting boundary conditions at the inlet representing a closed valve. 

These procedures have always led to a single real solution for the boundary value in all our numerical computations employing relevant parameters. After the boundary state $\mathbf U_L$ or $\mathbf U_R$ has been determined the corresponding boundary state in the variable $\mathbf V$ is derived from~\eqref{eq:laxcurves}.

\begin{figure}
  \centering
  \begin{tabular}{@{\hspace{1em}} c @{\hspace{6em}} m{\linewidth/2}}
    \includegraphics{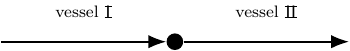}
    &
      \includegraphics{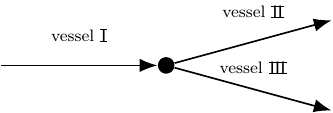}
\end{tabular}
\caption{One-to-one (left) and one-to-two (right) vascular junction.}\label{fig:vascularjunctions}
\end{figure}

\section{Coupling in a vascular network}\label{sec:cpl}
A full arterial network with bifurcations, curvatures and discontinuities in wall properties may be approximated by a graph with edges representing straight vessel segments. In addition, suitable boundary conditions at the nodes are required that connect the flow dynamics of the incoming and outgoing edges. In the following we discuss some of these \emph{coupling conditions} in the basic cases shown in Figure~\ref{fig:vascularjunctions} and embed them in the relaxation form \eqref{eq:relaxation} that is used within our scheme. We thereby focus on model~\eqref{eq:systemQ} and without loss of generality assume $P_{\text{ext}}=0$.

\subsection{One-to-one coupling}\label{sec:onetoone}
To model discontinuities in vessel properties as they might occur e.g.\ in the presence of a stent, the domain is typically decomposed and two segments with varying model parameters are coupled at an interface, as visualized in Figure~\ref{fig:vascularjunctions} (left). In this situation conditions are required to determine the boundary states $\mathbf U_R^n$ and $\mathbf V_R^n$ for the left coupled vessel as well as $\mathbf U_L^n$ and $\mathbf V_L^n$ for the right coupled vessel at the interface. These states are used for the update of the numerical solution at the interface. As in Section~\ref{sec:BC} we neglect the time index in the following.

Based on \cite{formaggia2003} we impose continuity of both, mass flux and generalized total pressure\footnote{The total pressure we consider here generalizes the form in \cite{formaggia2003}, where $\alpha=1$ is assumed.} defined by $p_t = \alpha \frac{\rho}{2} \bigl( \frac{Q}{A} \bigr)^2 + p$, i.e.\ the coupling conditions 
\begin{align}
	Q_R &= Q_L, \label{eq:contQ} \\
  \alpha \, \frac{\rho}{2} \left( \frac{Q_R}{A_R} \right)^2 + p\left( A_R; A_0^{\rmn{1}}, \beta^{\rmn{1}} \right) &= \alpha \, \frac{\rho}{2} \left( \frac{Q_L}{A_L} \right)^2 + p\left( A_L; A_0^{\rmn{2}}, \beta^{\rmn{2}} \right) \label{eq:contpt}
\end{align}
hold. Here, the reference section area and the wall parameter in the pressure law may depend on the vessel and we employ the notations $A_0^{\rmn{1}}$ and $\beta^{\rmn{1}}$ to refer to the parameters of the incoming vessel on the left and $A_0^{\rmn{2}}$ and $\beta^{\rmn{2}}$ to refer to the parameters of the outgoing vessel on the right, respectively. We note that in addition to the conditions that we consider here, extended conditions are available that take into account the coupling angle, see~\cite{formaggia2003}.

\paragraph{Coupling the relaxation system.}
In our relaxation approach coupling conditions for the auxiliary variables $\mathbf{V}$ are additionally required to provide all boundary values in scheme~\eqref{eq:limitscheme}. To numerically approximate the networked system using this scheme we alternate between computing the boundary/coupling data and updating the numerical solution on all edges. We employ the consistency principle introduced in \cite{herty2023centrschemtwo} to derive suitable conditions.

As the variable $V^A$ represents the first component of the flux $\mathbf F(\mathbf U)$, condition \eqref{eq:contQ} implies the consistent condition
	\begin{equation}\label{eq:VA}
          V_R^A = V_L^A.
	\end{equation}
Similarly, a condition for the variable $V^Q$ is derived using \eqref{eq:contpt}. Taking into account the flux in \eqref{eq:flow2conservative} we obtain
  \begin{equation}\label{eq:VQ}
A_R^{-1}\left[ V^Q_R - \frac{\alpha}{2} \frac{Q_R^2}{A_R} +   \rho^{-1} \, P(A_R; A_0^{\rmn{1}}, \beta^{\rmn{1}}) \right]= A_L^{-1}\left[ V^Q_L - \frac{\alpha}{2} \frac{Q_L^2}{A_L} + \rho^{-1}\,P(A_L; A_0^{\rmn{2}}, \beta^{\rmn{2}}) \right].
  \end{equation}
To eventually obtain the coupling data, a solution of the nonlinear system given by \eqref{eq:contQ}, \eqref{eq:contpt}, \eqref{eq:VA}, \eqref{eq:VQ} and \eqref{eq:laxcurves} needs to be computed. To this end we first eliminate the variables within $\mathbf V_R$ and $\mathbf V_L$ using \eqref{eq:laxcurves} in \eqref{eq:VA} and \eqref{eq:VQ} and then solve the remaining system for the variables within $\mathbf U_R$ and $\mathbf U_L$. We provide a detailed algorithm for the solution of the system in Appendix~\ref{sec:algo}.

\paragraph{Velocity formulation.}
In analogy to the above derivation, coupling conditions for the \textit{velocity form} and $\alpha=1$ are derived, which read
    \begin{align*}
        A_R u_R &= A_L u_L,\\
      \frac{\rho}{2} u_R^2 + p\left( A_R; A_0^{\rmn{1}}, \beta^{\rmn{1}} \right) &= \frac{\rho}{2} u_L^2 + p\left( A_L; A_0^{\rmn{2}}, \beta^{\rmn{2}} \right),\\
      V_R^A &= V_L^A,\\
      V_R^u &= V_L^u.
    \end{align*}
Again, solving this system taking \eqref{eq:laxcurves} into account gives rise to the boundary fluxes within scheme \eqref{eq:limitscheme} for the coupled system \eqref{eq:systemu}. 
	
\subsection{One-to-two coupling}
	This section is concerned with bifurcations i.e.\ vascular junctions, at which one vessel is divided into two as shown in Figure~\ref{fig:vascularjunctions} (right).
  In the following, we indicate for all quantities the corresponding vessel of the junction using an index from \Rmn{1} to \Rmn{3} as marked in the figure. Boundary/coupling data are sought at the right boundary for vessel \Rmn{1} and at the left boundaries for vessels \Rmn{2} and \Rmn{3}. Following \cite{formaggia2003} we impose conservation of mass flow and continuity of total pressure at the coupling node:
  \begin{align}
    Q_R^\rmn{1} &= Q_L^\rmn{2} + Q_L^\rmn{3},\label{eq:flowconservation}\\
    \alpha \, \frac{\rho}{2} \left( \frac{Q_R^\rmn{1}}{A_R^\rmn{1}} \right)^2 + p\left( A_R^\rmn{1}; A_0^{\rmn{1}}, \beta^{\rmn{1}} \right) &= \alpha \, \frac{\rho}{2} \left( \frac{Q_L^\rmn{2}}{A_L^\rmn{2}} \right)^2 + p\left( A_L^\rmn{2}; A_0^{\rmn{2}}, \beta^{\rmn{2}} \right),\\
    \alpha \, \frac{\rho}{2} \left( \frac{Q_R^\rmn{1}}{A_R^\rmn{1}} \right)^2 + p\left( A_R^\rmn{1}; A_0^{\rmn{1}}, \beta^{\rmn{1}} \right) &= \alpha \, \frac{\rho}{2} \left( \frac{Q_L^\rmn{3}}{A_L^\rmn{3}} \right)^2 + p\left( A_L^\rmn{3}; A_0^{\rmn{2}}, \beta^{\rmn{2}} \right).
  \end{align}
  
        In addition, we propose new conditions for the variable $\mathbf{V}$ to be used within scheme~\eqref{eq:limitscheme}. Relying again on the consistency principle in the relaxation limit gives rise to
        \begin{equation}\label{eq:VA12}
          V_R^{A,\rmn{1}}=V_L^{A,\rmn{2}} + V_L^{A,\rmn{3}} 
        \end{equation}
        when considering the original condition \eqref{eq:flowconservation}. In analogy to \eqref{eq:VQ} we further derive conditions for the variable $V^Q$ that read
        \begin{align}
          \frac{A_L^\rmn{2}}{A_R^\rmn{1}}\left[ V^{Q,\rmn{1}}_R - \frac{\alpha}{2} \frac{(Q_R^\rmn{1})^2}{A_R^\rmn{1}} +   \rho^{-1} \, P(A_R^\rmn{1}; A_0^{\rmn{1}}, \beta^{\rmn{1}}) \right]&= V^{Q,\rmn{2}}_L - \frac{\alpha}{2} \frac{(Q_L^\rmn{2})^2}{A_L^\rmn{2}} + \rho^{-1}\,P(A_L^\rmn{2}; A_0^{\rmn{2}}, \beta^{\rmn{2}}) , \label{eq:VQ121}\\
          \frac{A_L^\rmn{3}}{A_R^\rmn{1}}\left[ V^{Q,\rmn{1}}_R - \frac{\alpha}{2} \frac{(Q_R^\rmn{1})^2}{A_R^\rmn{1}} +   \rho^{-1} \, P(A_R^\rmn{1}; A_0^{\rmn{1}}, \beta^{\rmn{1}}) \right]&= V^{Q,\rmn{3}}_L - \frac{\alpha}{2} \frac{(Q_L^\rmn{3})^2}{A_L^\rmn{3}} + \rho^{-1}\,P(A_L^\rmn{3}; A_0^{\rmn{3}}, \beta^{\rmn{3}}). \label{eq:VQ122}
        \end{align}
        To compute the coupling data in practice, the full system is solved for the coupling states $\mathbf U_R^\rmn{1}$, $\mathbf U_L^\rmn{2}$ and $\mathbf U_L^\rmn{3}$ after component-wise substituting
        \begin{align*}
          \mathbf V_R^\rmn{1}= \mathbf V_N^\rmn{1}+ \lambda(\mathbf U_N^\rmn{1} - \mathbf U_R^\rmn{1}), \quad \mathbf V_L^\rmn{2}= \mathbf V_1^\rmn{2}+ \lambda(\mathbf U_L^\rmn{2} - \mathbf U_1^\rmn{2}), \quad \mathbf V_L^\rmn{3}= \mathbf V_1^\rmn{3}+ \lambda(\mathbf U_L^\rmn{3} - \mathbf U_1^\rmn{3}) 
        \end{align*}
        in \eqref{eq:VA12}, \eqref{eq:VQ122} and \eqref{eq:VQ122}. We use the multidimensional Newton--Raphson method to solve this nonlinear system.

\section{Numerical experiments}\label{sec:exp}

In this section we apply our numerical scheme \eqref{eq:limitscheme} in combination with the derived boundary and coupling conditions in various numerical experiments to demonstrate the performance of our approach. Experiments on one (the uncoupled case) and two edges (one-to-one coupling) are considered and each edge is discretized over uniform mesh cells of size $\Delta x$. Fixed time steps are used that are given by 
\[
  \Delta t = \CFL \frac{\Delta x}{\lambda},
  \]
where the relaxation speed $\lambda$ is chosen minimal with respect to \eqref{eq:lambdastability}. Each edge represents a 400 cm long vessel that is discretized over 800 grid points if not otherwise noted. Those unphysiologically long vessels have been chosen to better visualize the dynamics in the following. We fix $\CFL=1$ for the first order scheme resulting in the classical Lax--Friedrichs scheme; for our second order scheme $\CFL=0.2$ has lead to accurate results and robust computations.
The code is implemented in the Julia programming language \cite{Julia} on the basis of the implementation for the scheme \cite{CodeCentralNetworkScheme}. The main parameters are chosen as $A_0=6.6$ cm$^2$, $h_0=0.26$ cm, $\nu = 0.5$, $\mu=0$,  $E=2.43 \cdot 10^6$ dyne/cm$^2$, $\rho=1.06$ g/cm$^2$, and $\alpha=1$.

\subsection{The uncoupled scheme}
In this section we consider numerical experiments and tests of the uncoupled scheme, i.e.\ on a single edge. 

\paragraph{Grid convergence.}
We first study the convergence of our first and second order schemes from Section~\ref{sec: LF} in space. Therefore we conduct an experiment with smooth initial data and Neumann boundary conditions and another one simulating an entering pulse wave from the left boundary. Within the second order MUSCL scheme \eqref{eq:muscl} we have varied the Courant number with the time step taking $\CFL=0.2 \cdot \Delta x$ to avoid a reduction of the error due to the first order time discretization that we use.

\begin{table}
    \caption{$L^1$-errors at time $t=0.05$ s for the flow rate $Q$ and the section area $A$ under grid refinement in case of smooth initial data and Neumann boundary conditions.}\label{tab:NerrorsGauss}\vspace{1em}
  \centering
    \footnotesize
    \begin{tabular}{cc|cc|cc}
        &&\multicolumn{2}{c|}{first order scheme} & \multicolumn{2}{c}{second order scheme} \\
        Quantity & num.\ cells & $L^1$-error & EOC & $L^1$-error & EOC   \\
        \hline
           $Q$ & 50 & 1.931 &  & 15.86 & \\
            & 100 & 1.161 & 0.734 & 5.818 &  1.447\\
            & 200 & 0.633 & 0.874 & 1.795 & 1.696 \\
            & 400 & 0.318 & 0.996 & 0.500 & 1.844\\
            & 800 & 0.147 & 1.113 & 0.136 & 1.881\\
            & 1600 & 0.064 & 1.209 & 0.034 & 1.980\\
            \hline
          $A$  & 50 & 4.278e-3 &  & 2.943e-2 & \\ 
            & 100 & 2.624e-3 & 0.705 & 1.082e-2 & 1.444\\
            & 200 & 1.407e-3 & 0.900 & 3.297e-3 & 1.714\\
            & 400 & 6.975e-4 & 1.012 & 9.129e-4 & 1.853\\
            & 800 & 3.210e-4 & 1.119 & 2.469-4 & 1.886\\
            & 1600 & 1.384e-4 & 1.214 & 6.247e-5 & 1.983\\
    \end{tabular}
\end{table}

In the first experiment we take the section area $A(x, 0) = A_0 + \exp(-0.005(x-100)^2)$ and the mass flux $Q(\cdot, 0) \equiv 0$ on the shortened spatial domain $[0,200\text{ cm}]$. At the boundaries we impose homogeneous Neumann conditions. We successively refine the grid and compute discrete $L^1$-errors in the quantities $Q$ and $A$ at the final time $t=0.05$ with respect to a reference solution on $6400$ cells and present them in Table~\ref{tab:NerrorsGauss}. Along we show the experimental order of convergence (EOC)\footnote{The EOC is defined by EOC=$\log_2(e_1/e_2)$, where $e_1$ and $e_2$ are the errors in two consecutive lines of the table.}.

Table \ref{tab:NerrorsGauss} shows that the errors of the first order scheme indicate a nearly linear convergence rate. For the second order scheme we see nearly second order convergence for higher numbers of cells. Comparing both schemes we observe a clear advantage of the second order scheme as it achieves smaller errors on fine meshes. 

A similar grid convergence study is done in an experiment simulating a pulse wave entering the tube from the left-hand side. This pulse wave is achieved by imposing the sinusoidal inlet pressure
\[
P_{\nu}(t) = 6 \cdot 10^4 \cdot \sin(5 \pi t)
\]
at the left boundary.
Here, the initial values are taken constant as $A(\cdot, 0) \equiv A_0$ and $Q(\cdot, 0) \equiv 0$. Again a reference solution on 6400 cells is used to compute discrete $L^1$-errors at the final time $t=0.1$ s, which are presented in Table~\ref{tab:Nerrors}.

\begin{table}
    \centering
    \caption{$L^1$-errors at time $t=0.1$ s for the flow rate $Q$ and the section area $A$ under grid refinement in case of a pulse wave inflow.}\label{tab:Nerrors}\vspace{1em}
    \footnotesize
    \begin{tabular}{cc|cc|cc}
      &&\multicolumn{2}{c|}{first order scheme} & \multicolumn{2}{c}{second order scheme} \\
        Quantity & num. cells & $L^1$-error & EOC & $L^1$-error & EOC \\
        \hline
           $Q$ & 50 & 22.92 &  & 42.88& \\
            & 100 & 12.55 & 0.869 & 17.39 & 1.302\\
            & 200 & 6.417 & 0.968 & 6.236 & 1.479\\
            & 400 & 3.208 & 1.000  & 2.447 & 1.350\\
            & 800 & 1.535 & 1.063 & 1.012 & 1.274\\
            & 1600 & 0.675 & 1.186 & 0.406 & 1.319\\
            \hline
            $A$& 50 & 3.736e-2 &  & 7.162e-2 & \\
            & 100 & 2.085e-2 & 0.842 & 2.919e-2 & 1.295\\
            & 200 & 1.080e-2 & 0.949 & 1.055e-2 & 1.468\\
            & 400 & 5.450e-3 & 0.986 & 4.164e-3 & 1.342\\
            & 800 & 2.623e-3 & 1.055 & 1.728e-3 & 1.269\\
            & 1600 & 1.158e-3 & 1.180 & 6.938e-4 & 1.316\\
    \end{tabular}
\end{table}
Also in this setup the EOCs indicate a first order convergence of the $L^1$-error. The MUSCL scheme in combination with the second order boundary condition manages to further decrease the error; an EOC of approximately 1.3 can be observed. We note that due to the lack of smoothness of the left boundary data in time, second order convergence in space is not expected in this experiment. 

\begin{figure}
  \centering
  \includegraphics[width=0.8\linewidth]{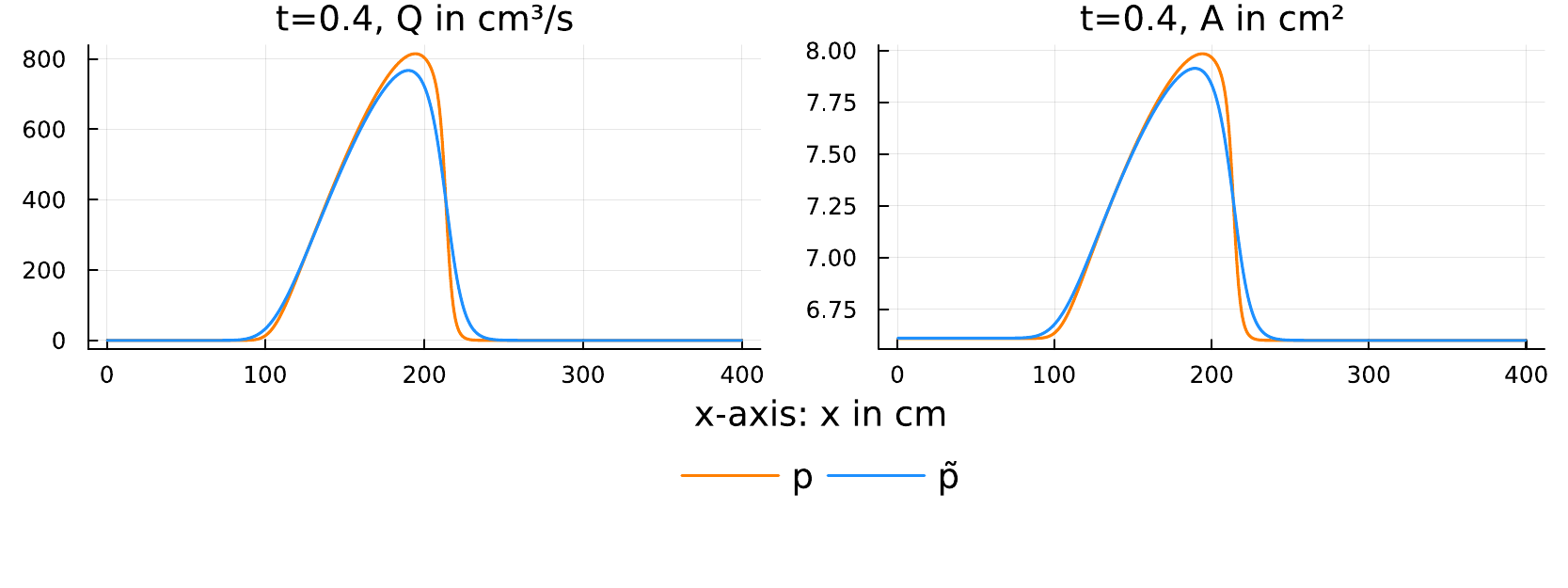}
  \caption{Flow rate and section area at time $t=0.4$ using the classical pressure model \eqref{eq:commonpressure} (orange) and the extended pressure model \eqref{eq:pressurelong} (blue).}
  \label{fig:aorta_comparePressure}
\end{figure}

\paragraph{Viscoelasticity.}
Next, we investigate the impact of the extended pressure \eqref{eq:parabolicterm} taking into account the viscoelasticity of the vessel wall. To this end we revisit the previous experiment considering the entering pulse wave and compare numerical solutions of the classical model ($\gamma=0$) and the extended pressure model (taking $\gamma = 120 \cdot \frac{A_0^{3/2}}{\sqrt{\pi}}$).
In Figure~\ref{fig:aorta_comparePressure} the numerical solution in terms of flow rate and section area at time instance $t=0.4$ are shown. A significant smoothing effect of the viscoelasticity is exhibited in both, the flow rate and the section area, decreasing the amplitude of the pressure wave. 

\subsection{One-to-one coupling}
In this section we consider two one-to-one coupling experiments studying discontinuities in different wall properties at the coupling node. In both experiments we use the initial data $A^\rmn{1}(\cdot, 0) \equiv A_0^\rmn{1}$, $A^\rmn{2}(\cdot, 0) \equiv A_0^\rmn{2}$ and $Q^\rmn{1}(\cdot, 0) \equiv Q^\rmn{2}(\cdot, 0) \equiv 0$, as boundary conditions we impose a pulse wave entering from the left boundary of vessel \Rmn{1} and homogeneous Neumann conditions at the right boundary of vessel \Rmn{2}. Both vessels are coupled at an interface as described in Section~\ref{sec:onetoone}

Firstly, a discontinuity in the reference section area $A_0$ is assumed; in more details we assume $A_0^\rmn{1}= 1.25 A_0$ on the left incoming vessel and $A_0^\rmn{2}=0.75A_0$ on the right outgoing vessel. In Figure~\ref{fig:aorta_coupling11_A0_ls} we show mass flow, pressure, section area and velocity $u=Q/A$ at two different time steps. At time $t=0.425$ we see some deflections at the coupling interface $x=0$. Due to the reduced vessel diameter the blood is decelerated in front of the interface. Thus, the velocity and the mass flow are decreased, and the pressure and section area are increased at this position. This effect creates backward-propagating waves, which are visible at time $t=0.6$ in addition to the first wave that has passed the interface.

\begin{figure}
  \centering
  \includegraphics[width=0.9\linewidth]{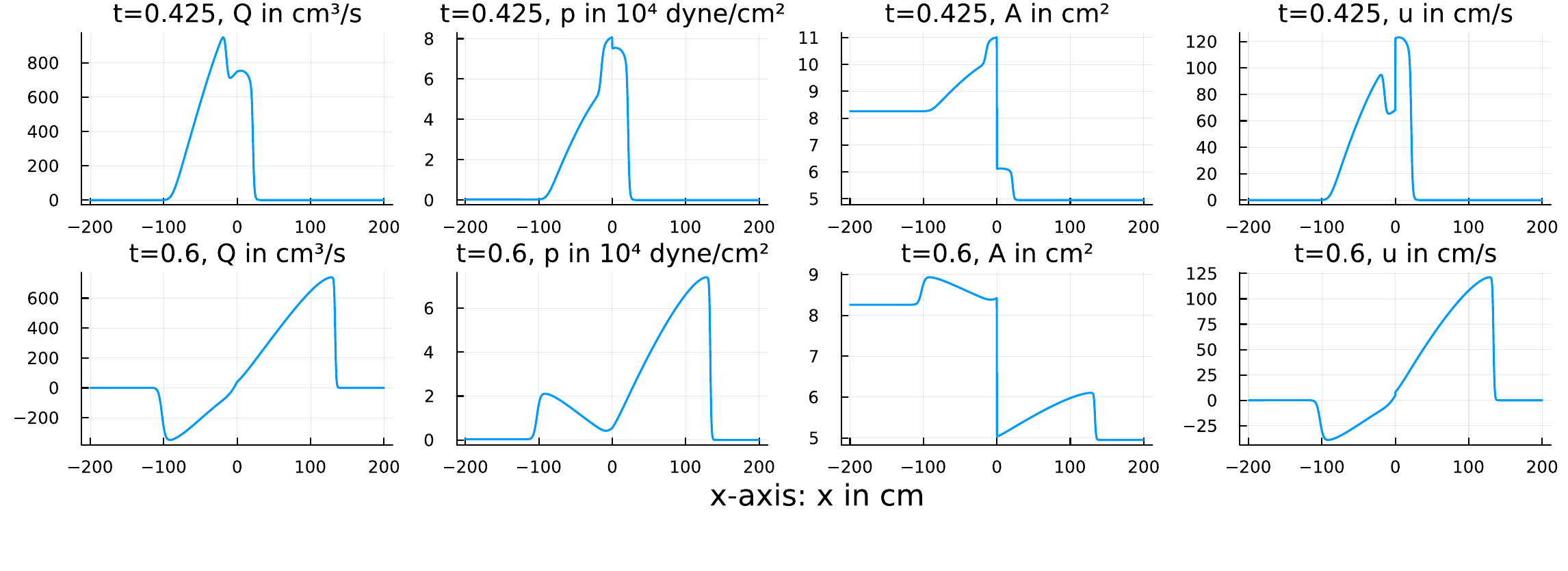}
  \caption{Flow rate, pressure, section area and velocity at time instances $t=0.425$ and $t=0.6$ over two vessels with different reference section areas ($A_0^\rmn{1}> A_0^\rmn{2}$) coupled at $x=0$.}
  \label{fig:aorta_coupling11_A0_ls}
\end{figure}

The next experiment is concerned with a discontinuity in the vessel elasticity. In more details a transition from the stiffer vessel \Rmn{1} to the more elastic vessel \Rmn{2} is modeled by imposing the Young modulus $E^\rmn{1} = 1.25 E$ to the left and  $E^\rmn{2} = 1.75 E$ to the right vessel. In Figure~\ref{fig:aorta_coupling11_E_ls} we show the numerical solution of this coupled experiment at two time instances. At time $t=0.4$ the more elastic vessel is further extended as can be seen examining the section area near the interface. As a consequence the blood velocity in front of the discontinuity is increased, whereas the pressure is decreased there and the mass flux is increased. At time $t=0.6$ the pressure wave has passed the interface and a backwards traveling wave appears from the interface causing a pressure decrease and a tightening of the vessel.

\begin{figure}
  \centering
  \includegraphics[width=0.9\linewidth]{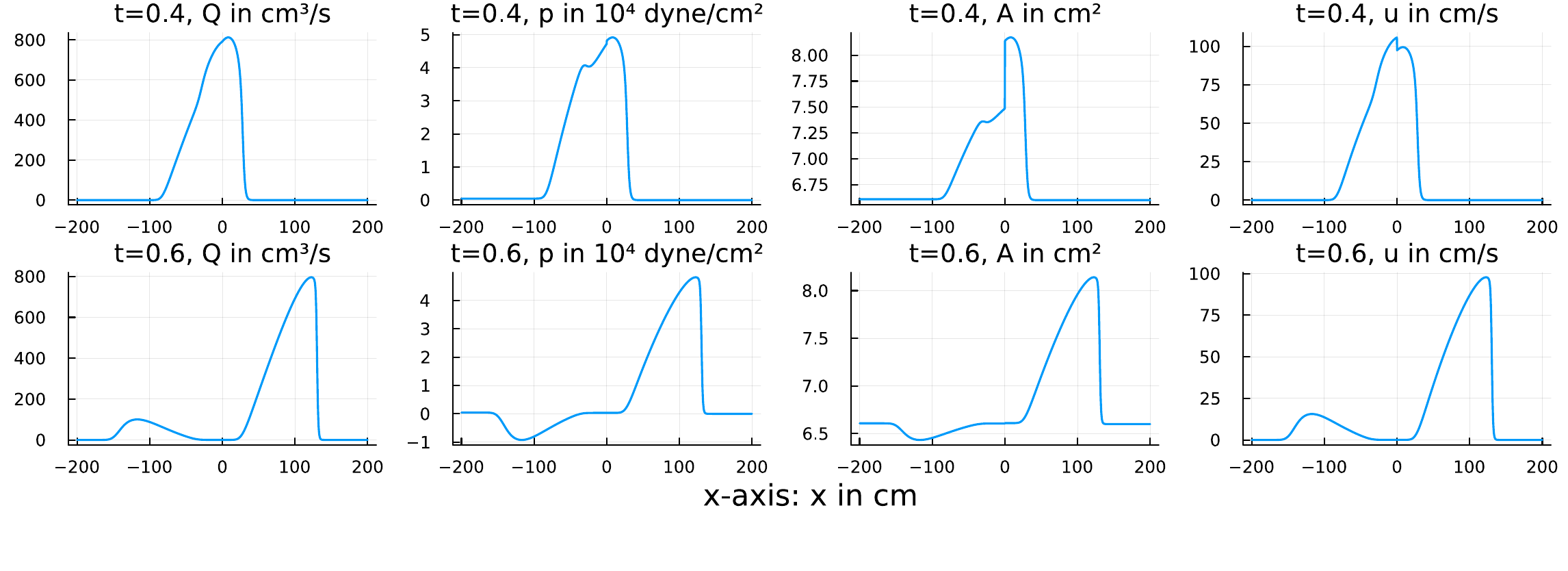}
  \caption{Flow rate, pressure, section area and velocity at time instances $t=0.4$ and $t=0.6$ over two vessels with different elasticity properties ($E^\rmn{1}< E^\rmn{2}$) coupled  at $x=0$.}
  \label{fig:aorta_coupling11_E_ls}
\end{figure}

\paragraph{Coupling error.} For consistency with the problem given by~\eqref{eq:systemQ} on both vessels together with the coupling conditions \eqref{eq:contQ} and \eqref{eq:contpt} it is necessary that our numerical solution approximately satisfies the coupling condition. We test this by computing the errors
\begin{equation}\label{eq:couplingerrors}
e_1^c = |Q_N^\rmn{1} - Q_1^\rmn{2}|, \quad e_2^c = \left| \alpha \, \frac{\rho}{2} \left(\frac{Q_N^\rmn{1}}{A_N^\rmn{1}} \right)^2 + p(A_N^\rmn{1}; A_0^\rmn{1}, \beta^\rmn{1}) - \alpha \, \frac{\rho}{2} \left(\frac{Q_1^\rmn{2}}{A_1^\rmn{2}} \right)^2 - p(A_1^\rmn{2}; A_0^\rmn{2}, \beta^\rmn{2}) \right|
\end{equation}
that quantify the approximation of the coupling conditions by substituting the numerical trace data next to the interface. In Table~\ref{tab:CplErrors} those errors are shown for both test cases above under grid refinement at the fixed time instance $t=0.5$, at which the pulse wave has reached the interface. In both experiments the EOCs indicate first order of convergence for both coupling errors confirming the consistency of our approach. 

\begin{table}
    \caption{Coupling errors defined in~\eqref{eq:couplingerrors} under grid refinement in the experiments considering discontinuous $A_0$ and $E$.}\label{tab:CplErrors} \vspace{.5em}
  \footnotesize
  \centering
  \begin{tabular}{c | c c c c | c c c c }
    &\multicolumn{4}{c|}{discontinuous $A_0$ } & \multicolumn{4}{c}{discontinuous $E$} \\
    $N$ & $e_1^c$ & EOC & $e_2^c$ & EOC & $e_1^c$ & EOC & $e_2^c$ & EOC\\
    \hline
    50 & 48.265 &  & 2047.383 &  & 57.677 &  & 5155.836 & \\ 
    100 & 24.411 & 0.983 & 1006.360 & 1.025 & 21.565 & 1.419 & 1909.133 & 1.433\\
    200 & 12.150 & 1.007 & 502.089 & 1.003 & 11.098 & 0.958 & 980.020 & 0.962\\
    400 & 6.055 & 1.005 & 250.877 & 1.001 & 5.559 & 0.997 & 490.107 & 1.000\\
    800 & 3.022 & 1.003 & 125.394 & 1.001 & 2.781& 0.999 & 244.958 & 1.001\\
    1600 & 1.509 & 1.002 & 62.685 & 1.000 & 1.391 & 1.000 & 122.462 & 1.000\\
  \end{tabular}

\end{table}

\section{Conclusion}
In this work, we have introduced a new 1D modeling approach to simulate blood flow in the cardiovascular system. Deriving the Lax--Friedrichs scheme from a model relaxation allows for a simple way to couple vessels in a network without relying on the eigenstructure of the regarding flow models. We have provided a second order scheme extension and showcased an adaptation of the approach to extended pressure models accounting for viscoelasticity of the vessel wall. Our numerical tests have confirmed consistency, efficiency and grid convergence of the proposed technique. The method, in particular, facilitates the extension of the one-dimensional modeling of arterial networks taking into account modified flows due to e.g.,\ occlusions, torsions or endovascular surgery. In \cite{herty2024} out method is applied in the modeling of aspiration therapy. Future work will focus on employing the method in the detailed modeling of blood clots with the aim to enable further insights into hemodynamics of stroke patients and the development of efficient operation planning.       

\section*{Funding}
The authors thank the Deutsche Forschungsgemeinschaft (DFG, German Research Foundation) for the financial support under 320021702/GRK2326 (Graduate College Energy, Entropy, and Dissipative Dynamics), through SPP 2410 (Hyperbolic Balance Laws in Fluid Mechanics: Complexity, Scales, Randomness) within the Projects 526006304 and 525842915, through SPP 2311 (Robust Coupling of Continuum-Biomechanical In Silico Models to Establish Active Biological System Models for Later Use in Clinical Applications – Co-Design of Modelling, Numerics and Usability) within the Project 548864771 and though Project 461365406 (New traffic models considering complex geometries and data). Support by the ERS Open Seed Fund of RWTH Aachen University through project OPSF781 is also acknowledged.

\section*{Conflict of interest}
The authors declare there is no conflict of interest.

\bibliographystyle{abbrvurl} 
\bibliography{references.bib}

\begin{thebibliography}{10}

\bibitem{balossino2008effec}
R.~Balossino, F.~Gervaso, F.~Migliavacca, and G.~Dubini.
\newblock Effects of different stent designs on local hemodynamics in stented arteries.
\newblock {\em J. Biomech.}, 41(5):1053--1061, 2008.
\newblock \href {https://doi.org/10.1016/j.jbiomech.2007.12.005} {\path{doi:10.1016/j.jbiomech.2007.12.005}}.

\bibitem{benemerito2023}
I.~Benemerito, A.~Mustafa, N.~Wang, A.~P. Narata, A.~Narracott, and A.~Marzo.
\newblock A multiscale computational framework to evaluate flow alterations during mechanical thrombectomy for treatment of ischaemic stroke.
\newblock {\em Front. Cardiovasc. Med.}, 10, Mar 2023.
\newblock \href {https://doi.org/10.3389/fcvm.2023.1117449} {\path{doi:10.3389/fcvm.2023.1117449}}.

\bibitem{Julia}
J.~Bezanson, A.~Edelman, S.~Karpinski, and V.~B. Shah.
\newblock Julia: A fresh approach to numerical computing.
\newblock {\em SIAM Rev.}, 59(1):65--98, 2017.
\newblock \href {https://doi.org/10.1137/141000671} {\path{doi:10.1137/141000671}}.

\bibitem{chen1994hyperconserlaws}
G.-Q. Chen, C.~D. Levermore, and T.-P. Liu.
\newblock Hyperbolic conservation laws with stiff relaxation terms and entropy.
\newblock {\em Comm. Pure Appl. Math.}, 47(6):787--830, June 1994.
\newblock \href {https://doi.org/10.1002/cpa.3160470602} {\path{doi:10.1002/cpa.3160470602}}.

\bibitem{crosetto2011fluid}
P.~Crosetto, P.~Reymond, S.~Deparis, D.~Kontaxakis, N.~Stergiopulos, and A.~Quarteroni.
\newblock Fluid-structure interaction simulation of aortic blood flow.
\newblock {\em Comput. \& Fluids}, 43:46--57, 2011.
\newblock \href {https://doi.org/10.1016/j.compfluid.2010.11.032} {\path{doi:10.1016/j.compfluid.2010.11.032}}.

\bibitem{donea1984time}
J.~Don\'{e}a, S.~Giuliani, H.~Laval, and L.~Quartapelle.
\newblock Time-accurate solution of advection-diffusion problems by finite elements.
\newblock {\em Comput. Methods Appl. Mech. Engrg.}, 45(1-3):123--145, 1984.
\newblock \href {https://doi.org/10.1016/0045-7825(84)90153-1} {\path{doi:10.1016/0045-7825(84)90153-1}}.

\bibitem{dubois1988boundconditnonlin}
F.~Dubois and P.~Le~Floch.
\newblock Boundary conditions for nonlinear hyperbolic systems of conservation laws.
\newblock {\em J. Differ. Equations}, 71(1):93--122, Jan. 1988.
\newblock \href {https://doi.org/10.1016/0022-0396(88)90040-X} {\path{doi:10.1016/0022-0396(88)90040-X}}.

\bibitem{fleeter2020multil}
C.~M. Fleeter, G.~Geraci, D.~E. Schiavazzi, A.~M. Kahn, and A.~L. Marsden.
\newblock Multilevel and multifidelity uncertainty quantification for cardiovascular hemodynamics.
\newblock {\em Comput. Methods Appl. Mech. Engrg.}, 365:113030, 37, 2020.
\newblock \href {https://doi.org/10.1016/j.cma.2020.113030} {\path{doi:10.1016/j.cma.2020.113030}}.

\bibitem{formaggia2003}
L.~Formaggia, D.~Lamponi, and A.~Quarteroni.
\newblock One-dimensional models for blood flow in arteries.
\newblock {\em J. Eng. Math.}, 47, 2003.
\newblock \href {https://doi.org/10.1023/B:ENGI.0000007980.01347.29} {\path{doi:10.1023/B:ENGI.0000007980.01347.29}}.

\bibitem{formaggia2006}
L.~Formaggia, D.~Lamponi, M.~Tuveri, and A.~Veneziani.
\newblock Numerical modeling of 1{D} arterial networks coupled with a lumped parameters description of the heart.
\newblock {\em Comput. Methods Biomech. Biomed. Eng.}, 9:273 -- 288, 2006.
\newblock \href {https://doi.org/10.1080/10255840600857767} {\path{doi:10.1080/10255840600857767}}.

\bibitem{formaggia1999multismodelcirculsystem}
L.~Formaggia, F.~Nobile, A.~Quarteroni, and A.~Veneziani.
\newblock Multiscale modelling of the circulatory system: A preliminary analysis.
\newblock {\em Comput. Visual Sci.}, 2(2-3):75--83, Dec. 1999.
\newblock \href {https://doi.org/10.1007/s007910050030} {\path{doi:10.1007/s007910050030}}.

\bibitem{forti2019transopticmonit}
R.~M. Forti, C.~G. Favilla, J.~M. Cochran, W.~B. Baker, J.~A. Detre, S.~E. Kasner, M.~T. Mullen, S.~R. Messé, W.~A. Kofke, R.~Balu, D.~Kung, B.~A. Pukenas, N.~I. Sedora-Roman, R.~W. Hurst, O.~A. Choudhri, R.~C. Mesquita, and A.~G. Yodh.
\newblock Transcranial optical monitoring of cerebral hemodynamics in acute stroke patients during mechanical thrombectomy.
\newblock {\em J. Stroke Cerebrovascular Dis.}, 28(6):1483--1494, 2019.
\newblock \href {https://doi.org/10.1016/j.jstrokecerebrovasdis.2019.03.019} {\path{doi:10.1016/j.jstrokecerebrovasdis.2019.03.019}}.

\bibitem{friedrichs1971system}
K.~O. Friedrichs and P.~D. Lax.
\newblock Systems of conservation equations with a convex extension.
\newblock {\em Proc. Nat. Acad. Sci. U.S.A.}, 68:1686--1688, 1971.
\newblock \href {https://doi.org/10.1073/pnas.68.8.1686} {\path{doi:10.1073/pnas.68.8.1686}}.

\bibitem{fung1994biomec}
Y.~C. Fung and S.~C. Cowin.
\newblock Biomechanics: Mechanical properties of living tissues, 2nd ed.
\newblock {\em J. Appl. Mech.}, 61(4):1007–1007, Dec. 1994.
\newblock \href {https://doi.org/10.1115/1.2901550} {\path{doi:10.1115/1.2901550}}.

\bibitem{herty2023centrschemtwo}
M.~Herty, N.~Kolbe, and S.~Müller.
\newblock A central scheme for two coupled hyperbolic systems.
\newblock {\em Commun. Appl Math. Comput.}, Nov. 2023.
\newblock \href {https://doi.org/10.1007/s42967-023-00306-5} {\path{doi:10.1007/s42967-023-00306-5}}.

\bibitem{herty2023centr}
M.~Herty, N.~Kolbe, and S.~Müller.
\newblock Central schemes for networked scalar conservation laws.
\newblock {\em Netw. Heterog. Media}, 18(1):310--340, 2023.
\newblock \href {https://doi.org/10.3934/nhm.2023012} {\path{doi:10.3934/nhm.2023012}}.

\bibitem{herty2024}
M.~Herty, N.~Kolbe, and M.~Neidlin.
\newblock A one-dimensional model for aspiration therapy in blood vessels, 2024.
\newblock \href {https://arxiv.org/abs/2403.05494} {\path{arXiv:2403.05494}}, \href {https://doi.org/10.48550/arXiv.2403.05494} {\path{doi:10.48550/arXiv.2403.05494}}.

\bibitem{hu2017asymppreserschem}
J.~Hu, S.~Jin, and Q.~Li.
\newblock Asymptotic-{{Preserving Schemes}} for {{Multiscale Hyperbolic}} and {{Kinetic Equations}}.
\newblock In {\em Handb. {{Numer. Anal.}}}, volume~18, pages 103--129. {Elsevier}, 2017.
\newblock \href {https://doi.org/10.1016/bs.hna.2016.09.001} {\path{doi:10.1016/bs.hna.2016.09.001}}.

\bibitem{hughes1973}
T.~J. Hughes and J.~Lubliner.
\newblock On the one-dimensional theory of blood flow in the larger vessels.
\newblock {\em Math. Biosci.}, 18(1–2):161–170, Oct 1973.
\newblock \href {https://doi.org/10.1016/0025-5564(73)90027-8} {\path{doi:10.1016/0025-5564(73)90027-8}}.

\bibitem{jin2012asympap}
S.~Jin.
\newblock Asymptotic preserving ({AP}) schemes for multiscale kinetic and hyperbolic equations: a review.
\newblock {\em Riv. Math. Univ. Parma (N.S.)}, 3(2):177--216, 2012.

\bibitem{jin1995relaxschemsystem}
S.~Jin and Z.~Xin.
\newblock The relaxation schemes for systems of conservation laws in arbitrary space dimensions.
\newblock {\em Comm. Pure Appl. Math.}, 48(3):235--276, 1995.
\newblock \href {https://doi.org/10.1002/cpa.3160480303} {\path{doi:10.1002/cpa.3160480303}}.

\bibitem{CodeCentralNetworkScheme}
N.~Kolbe.
\newblock Implementation of central schemes for networks of scalar conservation laws.
\newblock {\em GitHub repository}, 2022.
\newblock URL: \url{https://github.com/nklb/CentralNetworkScheme}.

\bibitem{kolbe2024numerschemcoupl}
N.~Kolbe, M.~Herty, and S.~M\"{u}ller.
\newblock Numerical schemes for coupled systems of nonconservative hyperbolic equations.
\newblock {\em SIAM J. Numer. Anal.}, 62(5):2143--2171, 2024.
\newblock \href {https://doi.org/10.1137/23M1615176} {\path{doi:10.1137/23M1615176}}.

\bibitem{kurganov2000newhighresol}
A.~Kurganov and E.~Tadmor.
\newblock New high-resolution central schemes for nonlinear conservation laws and convection-diffusion equations.
\newblock {\em J. Comput. Phys.}, 160(1):241--282, 2000.
\newblock \href {https://doi.org/10.1006/jcph.2000.6459} {\path{doi:10.1006/jcph.2000.6459}}.

\bibitem{liu1987hyperconserlawsrelax}
T.-P. Liu.
\newblock Hyperbolic conservation laws with relaxation.
\newblock {\em Commun.Math. Phys.}, 108(1):153--175, Mar. 1987.
\newblock \href {https://doi.org/10.1007/BF01210707} {\path{doi:10.1007/BF01210707}}.

\bibitem{lucca2023}
A.~Lucca, S.~Busto, L.~O. M\"uller, E.~F. Toro, and M.~Dumbser.
\newblock A semi-implicit finite volume scheme for blood flow in elastic and viscoelastic vessels.
\newblock {\em J. Comput. Phys.}, 495:Paper No. 112530, 42, 2023.
\newblock \href {https://doi.org/10.1016/j.jcp.2023.112530} {\path{doi:10.1016/j.jcp.2023.112530}}.

\bibitem{melis2019improved}
A.~Melis, F.~Moura, I.~Larrabide, K.~Janot, R.~Clayton, A.~Narata, and A.~Marzo.
\newblock Improved biomechanical metrics of cerebral vasospasm identified via sensitivity analysis of a 1d cerebral circulation model.
\newblock {\em J. Biomech.}, 90:24--32, 2019.
\newblock \href {https://doi.org/10.1016/j.jbiomech.2019.04.019} {\path{doi:10.1016/j.jbiomech.2019.04.019}}.

\bibitem{neidlin2016investhemo}
M.~Neidlin, S.~J. Sonntag, T.~Schmitz-Rode, U.~Steinseifer, and T.~A. Kaufmann.
\newblock Investigation of hemodynamics during cardiopulmonary bypass: A multiscale multiphysics fluid–structure-interaction study.
\newblock {\em Med. Eng. Phys.}, 38(4):380–390, Apr 2016.
\newblock \href {https://doi.org/10.1016/j.medengphy.2016.01.003} {\path{doi:10.1016/j.medengphy.2016.01.003}}.

\bibitem{nobile2008robin}
F.~Nobile and C.~Vergara.
\newblock An effective fluid-structure interaction formulation for vascular dynamics by generalized {R}obin conditions.
\newblock {\em SIAM J. Sci. Comput.}, 30(2):731--763, 2008.
\newblock \href {https://doi.org/10.1137/060678439} {\path{doi:10.1137/060678439}}.

\bibitem{peiro2009reduc}
J.~Peir\'{o} and A.~Veneziani.
\newblock Reduced models of the cardiovascular system.
\newblock In {\em Cardiovascular mathematics}, volume~1 of {\em MS\&A. Model. Simul. Appl.}, pages 347--394. Springer Italia, Milan, 2009.
\newblock \href {https://doi.org/10.1007/978-88-470-1152-6\_10} {\path{doi:10.1007/978-88-470-1152-6\_10}}.

\bibitem{pradhan2024}
A.~M. Pradhan, F.~Mut, and J.~R. Cebral.
\newblock A one-dimensional computational model for blood flow in an elastic blood vessel with a rigid catheter.
\newblock {\em Int. J. Numer. Methods Biomed. Eng.}, 40(7):e3834, 2024.
\newblock \href {https://doi.org/10.1002/cnm.3834} {\path{doi:10.1002/cnm.3834}}.

\bibitem{quarteroni2004mathem}
A.~Quarteroni and L.~Formaggia.
\newblock Mathematical modelling and numerical simulation of the cardiovascular system.
\newblock In {\em Handb. numer. anal. {V}ol. {XII}}, Handb. Numer. Anal., XII, pages 3--127. North-Holland, Amsterdam, 2004.
\newblock \href {https://doi.org/10.1016/S1570-8659(03)12001-7} {\path{doi:10.1016/S1570-8659(03)12001-7}}.

\bibitem{saini2021globalepidemstrok}
V.~Saini, L.~Guada, and D.~R. Yavagal.
\newblock Global epidemiology of stroke and access to acute ischemic stroke interventions.
\newblock {\em Neurology}, 97(20), Nov. 2021.
\newblock \href {https://doi.org/10.1212/wnl.0000000000012781} {\path{doi:10.1212/wnl.0000000000012781}}.

\bibitem{smith2002}
N.~P. Smith, A.~J. Pullan, and P.~J. Hunter.
\newblock An anatomically based model of transient coronary blood flow in the heart.
\newblock {\em SIAM J. Appl. Math.}, 62(3):990--1018, 2002.
\newblock \href {https://doi.org/10.1137/S0036139999355199} {\path{doi:10.1137/S0036139999355199}}.

\bibitem{thompson1987time}
K.~W. Thompson.
\newblock Time dependent boundary conditions for hyperbolic systems.
\newblock {\em J. Comput. Phys}, 68(1):1–24, Jan. 1987.
\newblock \href {https://doi.org/10.1016/0021-9991(87)90041-6} {\path{doi:10.1016/0021-9991(87)90041-6}}.

\bibitem{leer1979towarultimconserdifferschem}
B.~{van Leer}.
\newblock Towards the ultimate conservative difference scheme. {{V}}. {{A}} second-order sequel to {{Godunov}}'s method.
\newblock {\em J. Comput. Phys.}, 32(1):101--136, July 1979.
\newblock \href {https://doi.org/10.1016/0021-9991(79)90145-1} {\path{doi:10.1016/0021-9991(79)90145-1}}.

\bibitem{white1991viscous}
F.~White.
\newblock {\em Viscous Fluid Flow}.
\newblock McGraw-Hill series in mechanical engineering. McGraw-Hill, 1991.

\bibitem{wilkins2017europcardiovdiseasstatis}
E.~Wilkins, L.~Wilson, K.~Wickramasinghe, P.~Bhatnagar, J.~Leal, R.~Luengo-Fernandez, R.~Burns, M.~Rayner, and N.~Townsend.
\newblock {\em European Cardiovascular Disease Statistics 2017}.
\newblock EHN., Brussels, 2017.

\bibitem{xiaofei2015verif}
J.-M.~F. Xiaofei~Wang and P.-Y. Lagr\'ee.
\newblock Verification and comparison of four numerical schemes for a 1d viscoelastic blood flow model.
\newblock {\em Comput. Methods Biomech. Biomed. Eng.}, 18(15):1704--1725, 2015.
\newblock \href {https://doi.org/10.1080/10255842.2014.948428} {\path{doi:10.1080/10255842.2014.948428}}.

\end{thebibliography}
\appendix

\section{Boundary conditions of second order}\label{sec:boundary2}
To preserve the second order of approximation in space of our MUSCL approach \eqref{eq:muscl} also at the boundary we provide in this appendix a second order equivalent of condition  \eqref{eq:nonreflectingQ}.
  Therefore, we replace the first order difference in \eqref{eq:nonreflectingdifference} by a second order one. More specifically we propose to use a three point forward difference formula on the left boundary so that the condition 
  \begin{equation*}
    \mathbf l_1(\mathbf U_L)^T [4\mathbf F(\mathbf U_1) - 3\mathbf F(\mathbf U_L) -\mathbf F(\mathbf U_2)]  = 0
  \end{equation*}
  applies. Again, replacing the function evaluation with the variable $\mathbf{V}$ and using \eqref{eq:laxcurves} yields the boundary condition
  \begin{equation*}
    \begin{split}
      3\lambda \left( Q_1 - Q_L \right) + V_1^Q - V_2^Q = \left(\alpha \frac{Q_L}{A_L}+ \sqrt{\alpha (\alpha -1) \frac{Q_L^2}{A_L^2} + \frac{\beta}{2 \rho A_0} \sqrt{A_L}}~\right) \\
      \times \left( 3\lambda \left( A_1 - A_L \right) + V_1^A - V_2^A \right)
    \end{split}
  \end{equation*}
  for the system \eqref{eq:systemQ}.
  The corresponding boundary condition on the right boundary is derived using the analogue backward difference.

\section{Computation of the coupling data in the one-to-one coupling problem}\label{sec:algo}
In this appendix a procedure to solve the nonlinear system in Section~\ref{sec:onetoone} is described. For the sake of simplicity we assume $\alpha=1$. In the first step we eliminate the variables $\mathbf V_R$ and $\mathbf V_L$ in the system. To this end we exploit \eqref{eq:laxcurves} in \eqref{eq:VA} and \eqref{eq:VQ} and obtain
\begin{align}
  [\mathbf F^{\rmn{1}}(\mathbf U_N^{\rmn{1}})]_1+ \lambda \left(A_N^{\rmn{1}} -A_R \right) &= [\mathbf F^{\rmn{2}}(\mathbf U_1^{\rmn{2}})]_1 + \lambda \left( A_L - A_1^{\rmn{2}} \right), \label{eq:VAmod}\\
  [\mathbf F^{\rmn{1}}(\mathbf U_N^{\rmn{1}})]_2 + \lambda \left( Q_N^{\rmn{1}} - Q_R \right)&= \frac{\alpha}{2} \frac{Q_R^2}{A_R} - \frac{\alpha}{2} \frac{A_R Q_L^2}{A_L^2} + \frac{A_R}{A_L} ([\mathbf F^{\rmn{2}}(\mathbf U_1^{\rmn{2}})]_2+\lambda(Q_L-Q_1^{\rmn{2}})) \notag\\
                                        &\quad  +  \frac{1}{\rho} \Bigl(\frac{A_R}{A_L}  P(A_L; A_0^{\rmn{2}}, \beta^{\rmn{2}})  - P(A_R; A_0^{\rmn{1}}, \beta^{\rmn{1}}) \Bigr), \label{eq:VQmod} 
\end{align}
where in the numerical data near the interface and the flux function the corresponding edge is indicated by $\Rmn{1}$ or $\Rmn{2}$ and $[\mathbf v]_i$ refers to the $i$-th component of a given vector $\mathbf v$.

  It remains to solve the system given by \eqref{eq:contQ}, \eqref{eq:contpt}, \eqref{eq:VAmod} and \eqref{eq:VQmod} for the coupling data $A_L, Q_L, A_R, Q_R$; we note that all other variables are known parameters. To this end we first write $Q_L$ as function $Q_L(Q_R)=Q_R$ of $Q_R$ using \eqref{eq:contQ}, i.e.\ we replace $Q_L$ by $Q_R$ in the other equations. Next, \eqref{eq:VAmod} is used to rewrite $A_L$ as
  \begin{equation}
    A_L(A_R) = \frac{1}{\lambda}([\mathbf F^{\rmn{1}}(\mathbf U_N^{\rmn{1}})]_1 - [\mathbf F^{\rmn{2}}(\mathbf U_1^{\rmn{2}})]_1)+ A_N^{\rmn{1}} + A_1^{\rmn{1}} -A_R.
  \end{equation}
  These two new expressions are inserted into \eqref{eq:VQmod}, which results in a quadratic equation in $Q_R$ with coefficients depending only on $A_R$ that reads
  \begin{equation}\label{eq:Q}
		\begin{split}
                  \Bigl( \frac{1}{2} \frac{1}{A_R} - \frac{1}{2} \frac{A_R}{A_L(A_R)^2} \Bigr) Q_R^2  +  \bigl( \lambda + \frac{A_R}{A_L(A_R)} \lambda \bigr) Q_R  + \frac{A_R}{A_L(A_R)}([\mathbf F^{\rmn{2}}(\mathbf U_1^{\rmn{2}})]_2 - \lambda Q_1^{\rmn{2}}) \\
                  +  \frac{1}{\rho} \Bigl(\frac{A_R}{A_L} 
			P(A_L; A_0^{\rmn{2}}, \beta^{\rmn{2}}) - P(A_R; A_0^{\rmn{1}}, \beta^{\rmn{1}}) \Bigr) 
                  - [\mathbf F^{\rmn{1}}(\mathbf U_N^{\rmn{1}})]_2 - \lambda Q_N^{\rmn{1}}  =0.
		\end{split}
              \end{equation}
  If $A_L (A_R)=A_R$ the quadratic term in \eqref{eq:Q} vanishes and we have
              \[
                \begin{split}
                  Q^1_R(A_R) = \frac{1}{2\lambda} \left([\mathbf F^{\rmn{1}}(\mathbf U_N^{\rmn{1}})]_2 + \lambda Q_N^{\rmn{1}} - [\mathbf F^{\rmn{2}}(\mathbf U_1^{\rmn{2}})]_2 + \lambda Q_1^{\rmn{2}} \right)\\
                  -  \frac{1}{2\rho \lambda} \left(
                  P(A_L; A_0^{\rmn{2}}, \beta^{\rmn{2}}) - P(A_R; A_0^{\rmn{1}}, \beta^{\rmn{1}}) \right).
                \end{split}
              \]
  Otherwise, we obtain the two  solutions 
  \begin{equation*}
    \begin{split}
      Q_R^{1,2}(A_R) = - \frac{\lambda A_R A_L(A_R)}{A_L(A_R)-A_R} \pm \Biggl(\frac{\lambda^2 A_R^2 A_L(A_R)^2}{(A_L(A_R)-A_R)^2} + \frac{2 A_R A_L(A_R)^2}{A_L(A_R)^2 - A_R^2} \hspace{4em}\\
        \times \left( [\mathbf F^{\rmn{1}}(\mathbf U_N^{\rmn{1}})]_2 + \lambda Q_N^{\rmn{1}} - \frac{A_R}{A_L(A_R)}([\mathbf F^{\rmn{2}}(\mathbf U_1^{\rmn{2}})]_2 - \lambda Q_1^{\rmn{2}}) \right. \\
		\left.	 -\frac{1}{\rho} \Bigl(\frac{A_R}{A_L} P(A_L; A_0^{\rmn{2}}, \beta^{\rmn{2}}) - P(A_R; A_0^{\rmn{1}}, \beta^{\rmn{1}}) \Bigr)\right)\Biggr)^{1/2}.
		\end{split}
	\end{equation*}
Substituting either $Q^1_R(A_R)=Q_L$ or $Q^2_R(A_R)=Q_L$ into \eqref{eq:contpt} together with $A_L(A_R)$ gives rise to a nonlinear equation in $A_R$, which is solved numerically (e.g.\ using Newtons method). By re-substituting we then obtain the remaining coupling data. If multiple solutions to the system are obtained, we select the one that is closest to the trace data $\mathbf U_N^{\rmn{1}}$ and $\mathbf U_1^{\rmn{2}}$ in the sense of $L^1$-distance.
\end{document}